\documentclass[11pt]{amsart}
\usepackage[top=1.2in,bottom=1.2in,left=1.2in,right=1.2in]{geometry}

\usepackage{graphicx}
\usepackage{amsthm}
\usepackage{hyperref}
\hypersetup{
    colorlinks,
    citecolor=black,
    filecolor=black,
    linkcolor=black,
    urlcolor=black
}
\usepackage{amssymb}
\usepackage{epstopdf}
\usepackage{hyperref}
\usepackage{xypic}
\usepackage[textsize=tiny]{todonotes}
\usepackage{enumitem}
\usepackage{pinlabel}
\newcommand{\nc}[1]{\newcommand{#1}}
\nc{\dmo}[1]{\DeclareMathOperator{#1}}

\nc{\C}{\mathbb{C}}
\nc{\N}{\mathbb{N}}
\nc{\R}{\mathbb{R}}
\nc{\Z}{\mathbb{Z}}

\nc{\A}{\mathcal{A}}
\nc{\EFC}{\mathcal{EC}^\dagger}
\nc{\FC}{\mathcal{C}^\dagger}
\nc{\FA}{\mathcal{A}^\dagger}
\nc{\NA}{\mathcal{N\!A}^\dagger}
\nc{\LkA}{\mathcal{A}_{\textrm{Lk}}^\dagger}

\dmo{\Lk}{Lk}
\dmo{\Aut}{Aut}
\dmo{\Mod}{Mod}
\dmo{\Homeo}{Homeo}
\dmo{\Diff}{Diff}
\dmo{\PGL}{PGL}

\theoremstyle{plain}
\newtheorem{theorem}{Theorem}[section]
\newtheorem{proposition}[theorem]{Proposition}
\newtheorem{lemma}[theorem]{Lemma}
\newtheorem{corollary}[theorem]{Corollary}

\newcommand{\p}[1]{\bigskip \noindent \emph{#1}.}

\title{Automorphisms of the fine curve graph}

\author{Adele Long}
\author{Dan Margalit}
\author{Anna Pham}
\author{Yvon Verberne}
\author{Claudia Yao}

\address{Adele Long\\ Smith College\\ 1 Chapin Way\\ Unit 8281\\ Northampton, MA 01063}
\email{AdeleLRLong@gmail.com}

\address{Dan Margalit \\ School of Mathematics\\ Georgia Institute of Technology \\ 686 Cherry St. \\ Atlanta, GA 30332}
\email{margalit@math.gatech.edu}

\address{Anna Pham \\ School of Natural Sciences \& Mathematics\\ The University of Texas at Dallas \\ 800 W Campbell Rd \\ Richardson, TX 75080}
\email{annaphamutd@gmail.com}

\address{Yvon Verberne \\ School of Mathematics\\ Georgia Institute of Technology \\ 686 Cherry St. \\ Atlanta, GA 30332}
\email{yverberne3@gatech.edu}

 \address{Claudia Yao \\ Department of Mathematics \\ The University of Chicago \\ 5734 S.~University Avenue \\ Chicago, IL 60637}
\email{wxyao10@outlook.com}

\thanks{This material is based upon work supported by the National Science Foundation under Grant Nos. DMS-181843 and DMS-1811941.  The fourth author is partially supported by an NSERC--PDF Fellowship.}

\begin{document}

\maketitle

\begin{abstract}
Building on work of Farb and the second author, we prove that the group of automorphisms of the fine curve graph for a surface is isomorphic to the group of homeomorphisms of the surface.  This theorem is analogous to the seminal result of Ivanov that the group of automorphisms of the (classical) curve graph is isomorphic to the extended mapping class group of the corresponding surface.
\end{abstract}

%%%
%%%
%%%

\section{Introduction}

The fine curve graph $\FC(S)$ was recently introduced by Bowden--Hensel--Webb as a combinatorial tool for studying $\Homeo(S)$, the group of homeomorphisms of a surface $S$.  Its vertices are essential simple closed curves in $S$ and the edges are pairs of disjoint curves.  There are two versions of $\FC(S)$ in the literature, according to whether the curves are smooth or topological.  In this paper we take the vertices to be topological curves.  

Let $S_g$ be the closed, connected, orientable surface of genus $g$. Our main theorem is that the group of simplicial  automorphisms of $\FC(S_g)$ is isomorphic to $\Homeo(S_g)$ when $g \geq 2$; in other words, $\FC(S_g)$ is a combinatorial model for $\Homeo(S_g)$.  More precisely, we have the following statement.

\begin{theorem}
\label{thm:main}
For $g \geq 2$ the natural map
\[
\eta : \Homeo(S_g) \to \Aut \FC(S_g)
\]
is an isomorphism.
\end{theorem}

Theorem~\ref{thm:main} should be viewed as an analogue of the celebrated theorem of Ivanov \cite[Theorem 1]{Ivanov} that the group of automorphisms of the curve graph for $S_g$ is isomorphic to the mapping class group of $S_g$ when $g \geq 3$.  There are some immediate complications that arise for the fine curve graph that distinguish it from the curve graph.  To begin, the graph $\FC(S)$ has uncountably many vertices, and is even locally uncountable.  Moreover, two vertices of $\FC(S)$ can bound (countably many) bigons and can intersect along (uncountably many) intervals, etc.  The main difficulty in our work is to overcome these topological pathologies.

There is a precursor to Theorem~\ref{thm:main} that we use in our proof.  Specifically, Farb and the second named author studied what we presently refer to as the \emph{extended fine curve graph} $\EFC(S)$.  The vertices of $\EFC(S)$ are all simple closed curves in $S$, including the inessential ones, and the edges again are pairs of disjoint curves.  Farb and the second author proved the following theorem.

\begin{theorem}[Farb--Margalit]
\label{thm:efc}
For any surface $S$ without boundary, the natural map
\[
\nu : \Homeo(S) \to \Aut \EFC(S)
\]
is an isomorphism.

\end{theorem}

We give the original (unpublished) proof of Theorem~\ref{thm:efc} in Section~\ref{sec:efc}.  In the unpublished preprint of Farb and the second author \cite{FM}, Theorem~\ref{thm:efc} is stated more generally, where the vertices of the graph are locally flat $(n-1)$-spheres in an $n$-manifold $M$ and edges are for disjointness (with the automorphism group being $\Homeo(M)$).  The proof is essentially the same in this greater generality.  We refer the reader to Farb's lecture for more details \cite{Farb}.

%\p{Other reconstruction problems} The essential content of Theorem~\ref{thm:main} is that we can reconstruct an element of $\Homeo(S)$ from the a priori weaker information given by an element of $\Aut \FC(S)$.  There are many other examples of reconstruction problems in the literature, besides the result of Ivanov mentioned above.  For instance, Jeffers proves that a bijection of Euclidean $n$-space that preserves colinearity of triples of points must be an affine map \cite{Jeffers}.  He proves a similar theorem for hyperbolic $n$-space, the real projective plane, and the 2-sphere, among other results.   Another classical example of a reconstruction theorem is the fundamental theorem of projective geometry, which states that any simplicial automorphism of the Tits building for a finite dimensional vector space $V$ (the poset of proper, nontrivial subspaces of $V$) is induced by a semi-linear automorphism of $V$; see \cite{Putman}.  Theorem~\ref{thm:main} is an analogue of these results for topological surfaces.
%

\p{Prior results on the fine curve graph} As mentioned, $\FC(S)$ was introduced by Bowden--Hensel--Webb \cite{BHW}.  Using the smooth version of $\FC(S_g)$, they show that $\Diff_0(S_g)$, the identity component of the group of diffeomorphisms of $S_g$, admits many unbounded quasi-morphisms.  It follows that $\Diff_0(S_g)$ is not uniformly perfect and its fragmentation norm is unbounded.

Bowden--Hensel--Mann--Militon--Webb studied the dynamics of the action of $\Homeo(S)$ on $\FC(S)$ \cite{BHMMW}.  They proved that some elements of $\Homeo(S)$ act parabolically and that asymptotic translation length is a continuous function on $\Homeo(S)$.  They also characterized the elements of $\Homeo(T^2)$ that act hyperbolically on $\FC(T^2)$ in terms of rotation sets.

\p{Outline of the proof of Theorem~\ref{thm:main}} In order to prove Theorem~\ref{thm:main}, we construct an inverse map $\Aut \FC(S_g) \to \Homeo(S_g)$.  Because Theorem~\ref{thm:efc} already gives a map $\Aut \EFC(S_g) \to \Homeo(S_g)$ we can construct our inverse as a composition
\[
\Aut \FC(S_g) \to \Aut \EFC(S_g) \to \Homeo(S_g).
\]
So, besides giving the proof of Theorem~\ref{thm:efc}, our main task is to construct a homomorphism $\epsilon : \Aut \FC(S_g) \to \Aut \EFC(S_g)$.  In other words, given an automorphism $\alpha$ of $\FC(S_g)$ and an inessential curve $e$ in $S_g$, we need to associate another inessential curve $\hat \alpha(e)$ in a natural way.  To this end, we associate to each such $e$ a pair of vertices $\{c,d\}$ of $\FC(S_g)$, called a bigon pair, and define $\hat \alpha(e)$ to be the inessential curve associated to the bigon pair $\{\alpha(c),\alpha(d)\}$.  See the right-hand side of Figure~\ref{fig:pairs} for an example of a bigon pair.  This definition requires us to prove that bigon pairs are preserved by automorphisms of $\FC(S_g)$, which is the content of Proposition~\ref{prop:curve pairs}, the main technical result of the paper.

The proof of Theorem~\ref{thm:efc} mirrors the proof of Theorem~\ref{thm:main}: we use certain collections of vertices---convergent sequences of curves---to encode points in a surface $S$ in order to define a map $\Aut \EFC(S) \to \Homeo(S)$.  As such, convergent sequences play the role in the proof of Theorem~\ref{thm:efc} that bigon pairs play in the proof of Theorem~\ref{thm:main}.

\p{The torus case} As defined, the graph $\FC(T^2)$ is not connected since disjoint curves in $T^2$ lie in the same homotopy class.  In fact, the connected components precisely correspond to the homotopy classes of essential curves.  Since all of these components are isomorphic, it follows that every permutation of the components is induced by some element of $\Aut \FC(T^2)$.  On the other hand, $\Homeo(T^2)$ preserves the geometric intersection number between components, so the set of permutations of components arising from $\Homeo(T^2)$ is countable (this set of permutations is isomorphic to $\PGL_2(\Z)$).  Thus, $\Aut \FC(T^2)$ properly contains $\Homeo(T^2)$ as a subgroup of uncountable index.

Bowden--Hensel--Webb give a modified definition of $\FC(T^2)$, where the edges connect curves that intersect at most once.  It seems plausible that with this definition the natural map $\Homeo(T^2) \to \Aut \FC(T^2)$ is an isomorphism.  However, our arguments for Theorem~\ref{thm:main} do not apply, since they rely heavily on the fact that edges correspond to disjointness.  If one can show that an element of $\Aut \FC(T^2)$ preserves the set of edges corresponding to disjoint curves, then it would be possible to apply many of our arguments to the torus case.

\p{Outline of the paper} We begin in Section~\ref{sec:char} by showing that automorphisms of $\FC(S_g)$ preserve certain configurations of curves in $S_g$.  Specifically, these are the aforementioned bigon pairs, which are used to define the map $\epsilon$ discussed above, and sharing pairs, which are used to show that $\epsilon$ is well defined.  In Section~\ref{sec:fa}, we prove that three different fine arc graphs are connected.  The last of these, the fine linked arc graph $\LkA(S)$, is again used to show that $\epsilon$ is well defined.  In Section~\ref{sec:efc} we give the original proof of Theorem~\ref{thm:efc}.  Finally, in Section~\ref{sec:pf} we assemble the preceding results to prove Theorem~\ref{thm:main}.

\p{Acknowledgments} We would like to thank Jonathan Bowden, Sebastian Hensel, Kathryn Mann, Emmanuel Militon, and Richard Webb for a helpful conversation.  We are particularly grateful to Kathryn Mann for suggesting that we characterize bigons between vertices of the fine curve graph.  We would also like to thank Benson Farb for sharing his unpublished work, joint with the second author, and also for comments on an earlier draft.  We are grateful to Sam Taylor for explaining to us the proof of the connectivity of the fine curve graph.  We would like to thank Andy Putman and Roberta Shapiro for comments on an earlier draft.  This work was completed as part of the REU program at Georgia Institute of Technology; we would like to thank the National Science Foundation and the School of Mathematics at Georgia Institute of Technology for their generous support.

%%%
%%%
%%%

\section{Characterizations of curve configurations}
\label{sec:char}

The goal of this section is to prove that two specific types of configurations of curves in a surface $S$ are preserved under automorphisms of $\FC(S)$.  The corresponding statements are Propositions~\ref{prop:curve pairs} and~\ref{prop:curve pair pairs}.  We restrict ourselves in this section to the case $g \geq 2$.  

We begin with some preliminaries.  We say that curves $c$ and $d$ are \emph{noncrossing} at a component $a$ of $c \cap d$ if there is a neighborhood $U$ of $a$ and a homeomorphism $U \to \R^2$ so that the image of $c \cap U$ and $d \cap U$ lie in the (closed) upper and lower half-planes of $\R^2$, respectively.  The curves $c$ and $d$ are noncrossing if they are noncrossing at each component of $c \cap d$.

\begin{figure}
\includegraphics[scale=.35]{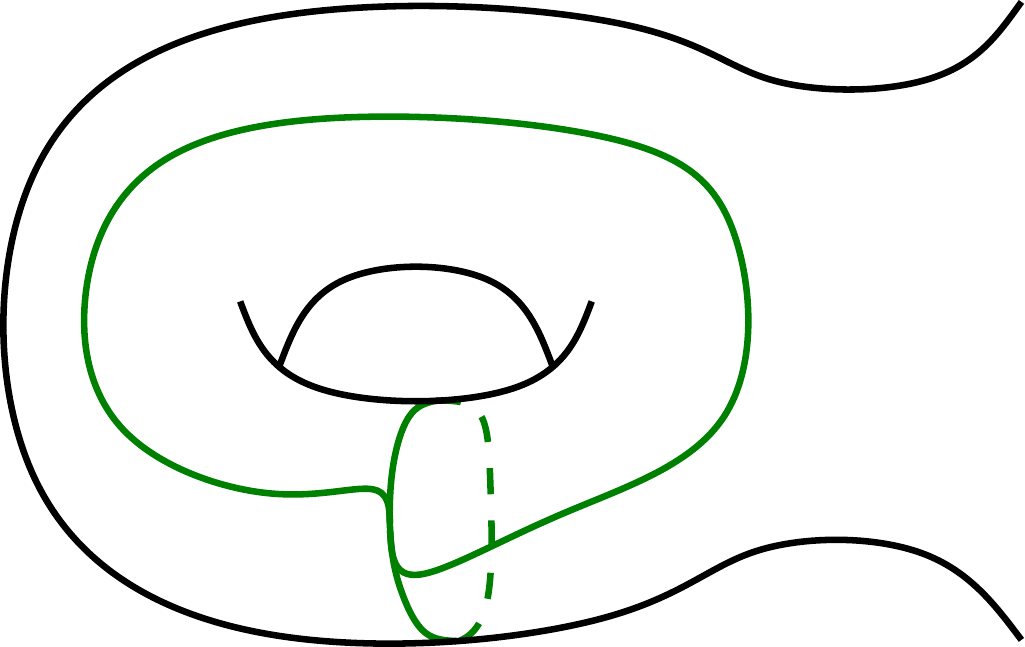}
\hspace*{.25in}
\raisebox{0.07\height}{\includegraphics[scale=.30]{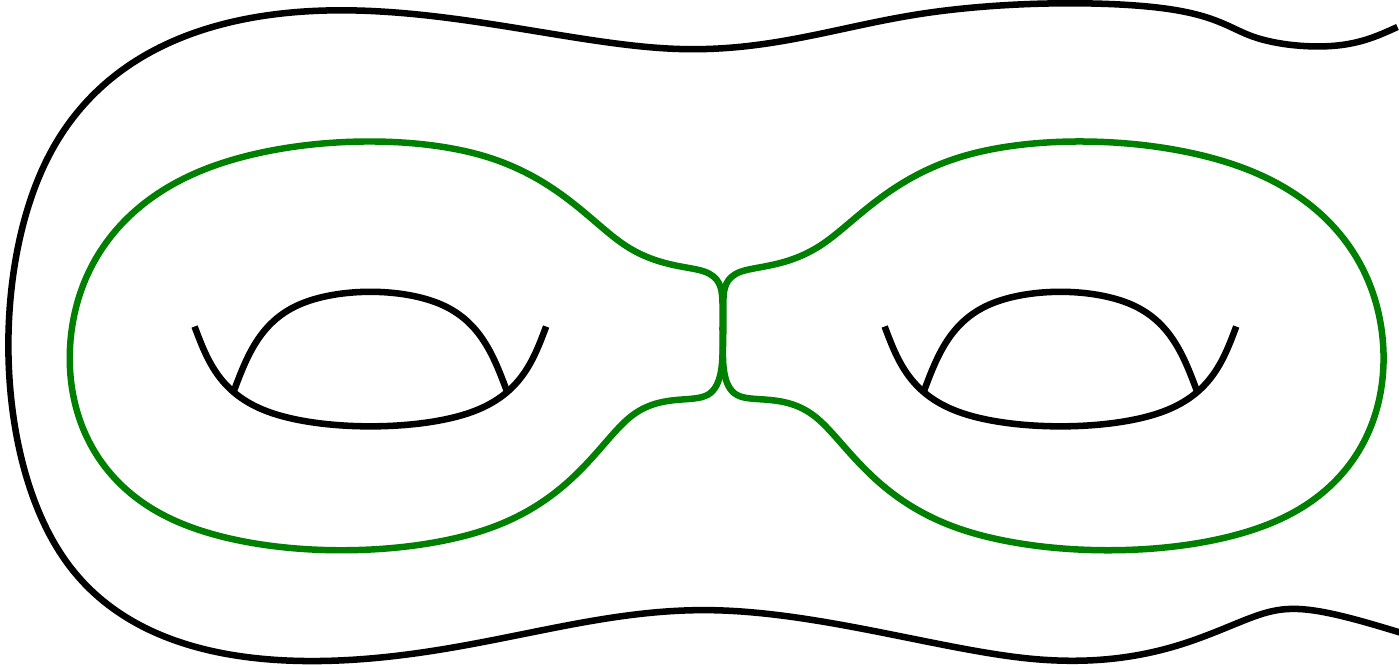}}
\hspace*{.25in}
\includegraphics[scale=.35]{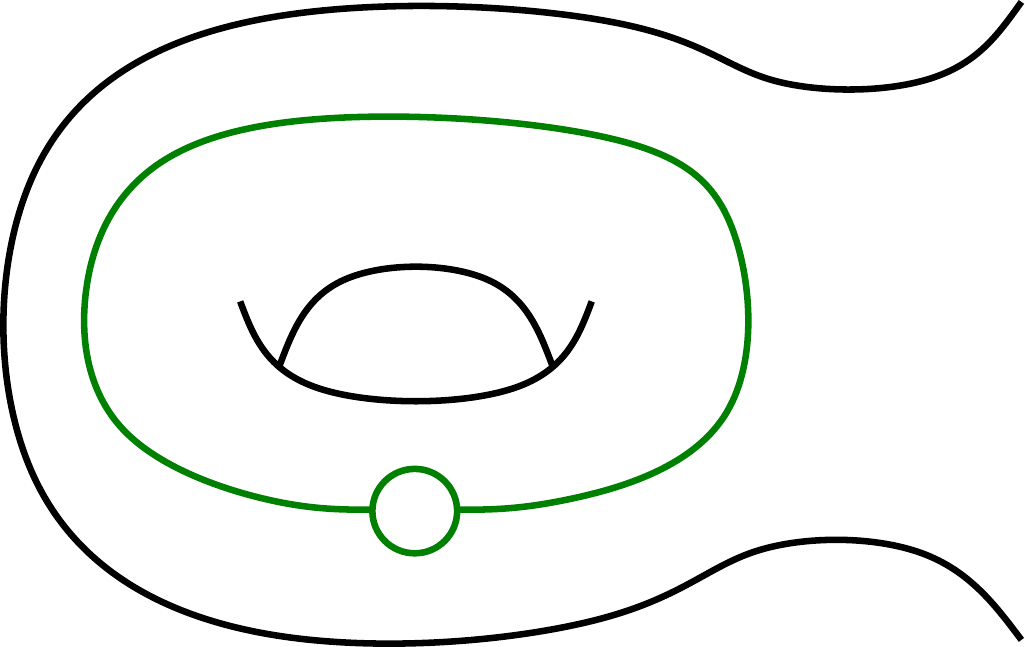}
\caption{\emph{Left to right:} a torus pair, a pants pair, and a bigon pair}
\label{fig:pairs}
\end{figure}

Next, we say that a pair of vertices $\{c,d\}$ in $\FC(S)$ is...
\begin{itemize}
\item a \emph{torus pair} if $c \cap d$ is a single interval and $c$ and $d$ cross at that interval, 
\item a \emph{pants pair} if $c \cap d$ is a single interval, $c$ and $d$ do not cross at that interval, and $c$ and $d$ are not homotopic, and 
\item a \emph{bigon pair} if $c \cap d$ is a nontrivial closed interval and $c$ and $d$ are homotopic.
\end{itemize}
If $c$ and $d$ form a torus pair, then there is a neighborhood of $c \cup d$ that is a torus with boundary.  If $c$ and $d$ form a pants pair, then there is a neighborhood of $c \cup d$ that is a pair of pants.  We say that a torus pair or pants pair $\{c,d\}$ is \emph{degenerate} if $c \cap d$ is a single point.  See Figure~\ref{fig:pairs} for pictures of the three types of pairs; in each case, we show the union of the two curves in the pair.  Given the union of a torus pair or a pants pair, there are three ways to write it as a union of two simple closed curves; in each case, the three pairs differ by a homeomorphism of $S$.

If $\{c,d\}$ is a nondegenerate torus pair (or pants pair) in $S$, then there is exactly one other essential curve $e$ contained in $c \cup d$; the curve $e$ is the closure in $S$ of the symmetric difference $c \triangle d$.  We also refer to $\{c,d,e\}$ as a \emph{torus triple} (or pants triple), since any two elements of the triple form a torus pair determining the third.

If $c$ and $d$ form a bigon pair, then $c$ and $d$ determine an inessential simple closed curve $e$; specifically, $e$ is the closure of the symmetric difference $c \triangle d$.  When the two curves in a bigon pair are nonseparating, we call the pair a \emph{nonseparating bigon pair}.

\begin{proposition}
\label{prop:curve pairs}
Let $g \geq 2$.  Then every automorphism of $\FC(S_g)$ preserves the set of nonseparating bigon pairs.
\end{proposition}

\begin{figure}
\includegraphics[scale=.35]{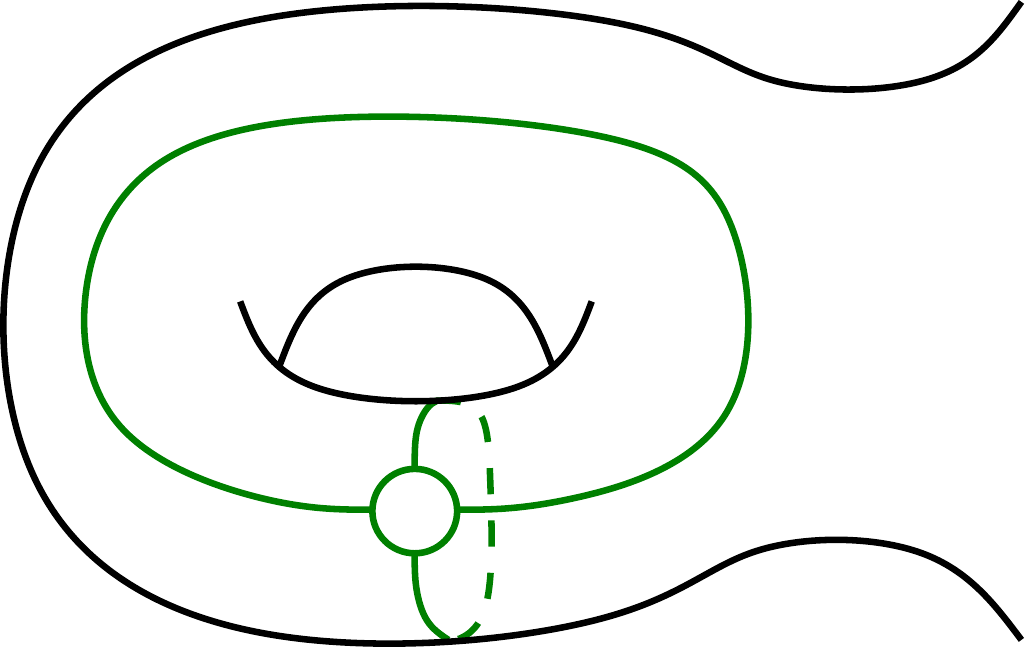}
\hspace*{.5in}
\includegraphics[scale=.35]{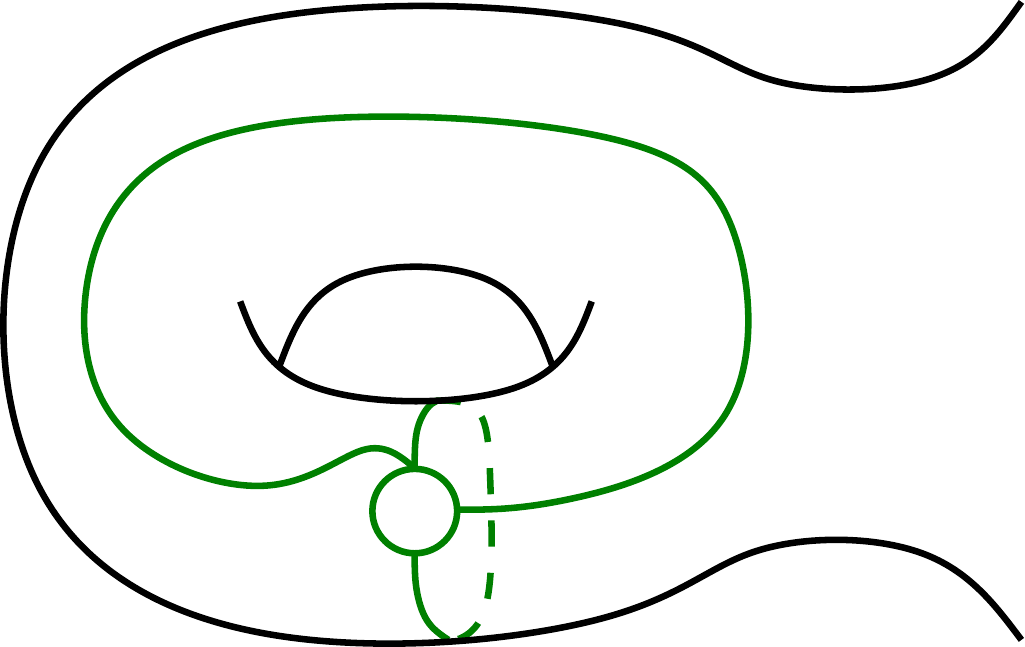}
\hspace*{.5in}
\includegraphics[scale=.35]{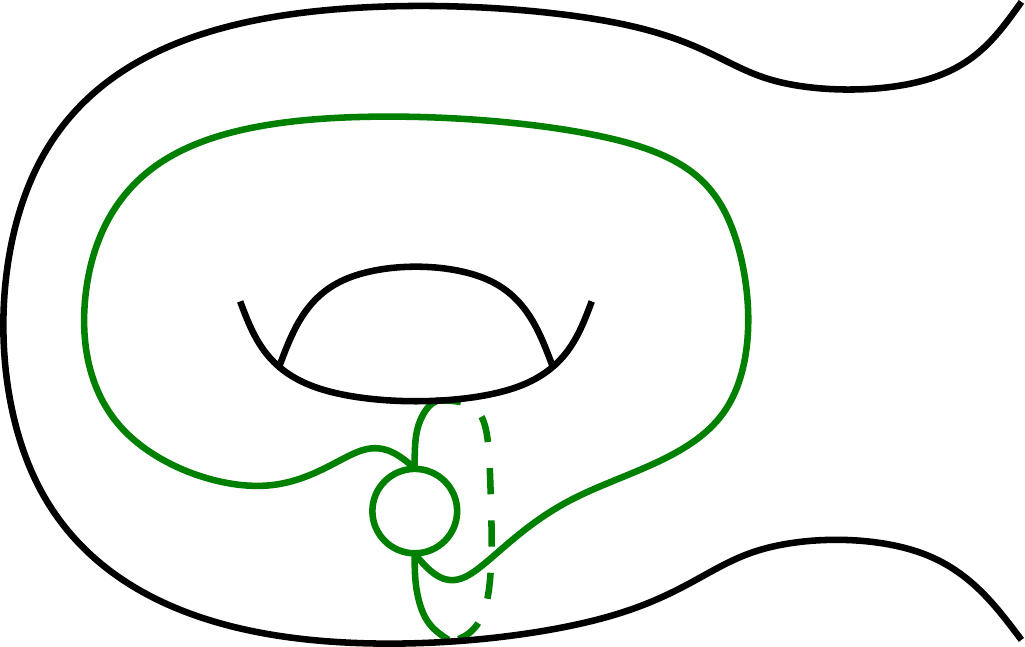}
\caption{The three types of sharing pairs}
\label{fig:sharing}
\end{figure}

For the second proposition, we require another definition.  Suppose that the bigon pairs $\{a,b\}$ and $\{a',b'\}$ determine the same inessential curve $e$.  In this case, each bigon pair gives rise to a single arc in the surface obtained by deleting the interior of the disk bounded by $e$; we identify this surface with $S_g^1$.  We say that the pair of bigon pairs $\{\{a,b\},\{a',b'\}\}$ is a \emph{sharing pair} if the corresponding arcs in $S_g^1$ have disjoint interiors.  We further say that the sharing pair is \emph{linked} if these two arcs are linked at $e$, which means that all boundary parallel curves in $S_g^1$ sufficiently close to the boundary intersect the two arcs alternately.  We note that if a sharing pair is linked, then all four of the corresponding curves must be nonseparating.  See Figure~\ref{fig:sharing} for pictures of sharing pairs; there are three configurations according to how many endpoints of the arcs agree at $e$.

\begin{proposition}
\label{prop:curve pair pairs}
Let $g \geq 2$.  Then every automorphism of $\FC(S_g)$ preserves the set of linked sharing pairs.
\end{proposition}

We prove Propositions~\ref{prop:curve pairs} and~\ref{prop:curve pair pairs} in Sections~\ref{sec:curve pairs} and~\ref{sec:curve pair pairs}, respectively.  Section~\ref{sec:transverse} contains a preliminary result, Lemma~\ref{lem:torus}, which states that automorphisms preserve torus pairs.

\p{Tame curves, wild pairs} Before we begin in earnest, we make some comments about the point-set topological issues that arise in this work.  First, it is a fact that every curve in a surface is tame in the sense that it is flat at each point.  It follows, for example, that any two nonseparating curves in a surface $S$ differ by a homeomorphism of $S$.  This can be thought of as a version of the change of coordinates principle in the theory of mapping class groups \cite[Section 1.3]{Primer}.

While curves themselves are tame, pairs of curves can exhibit complicated behavior.  If $a$ and $b$ are curves in a surface $S$, then $a \setminus b$ can be regarded as an open set in $S^1$, and hence is a countable union of disjoint intervals.  On the other hand, $a \cap b$ is a compact set, but it can be complicated.  The components of $a \cap b$ are (possibly degenerate) intervals, but there can be uncountably many; for instance $a \cap b$ can be a Cantor set.  

\subsection{Torus pairs}
\label{sec:transverse}

In this section we prove Lemma~\ref{lem:torus}, which states that automorphisms of $\FC(S_g)$ preserve the set of torus pairs, the set of (non)degenerate torus pairs, and the set of torus triples.  Along the way, we prove two auxiliary lemmas, Lemmas~\ref{lemma:joins} and~\ref{lem:hulls}.  In what follows, we say that two curves (or vertices of $\FC(S)$) \emph{intersect} if they are not disjoint.

\p{Sides} For the first lemma, a \emph{multicurve} is a finite collection of pairwise disjoint essential simple closed curves in $S$ (a curve is an example of a multicurve).  As such, multicurves are the same as finite cliques in $\FC(S)$.  A multicurve in $S$ is separating if its complement has more than one component.  We say that two curves $a$ and $b$ lie on the \emph{same side} of a separating multicurve $m$ if they are disjoint from $m$ and lie in the same complementary component.  

In the proof, we say that a graph is a \emph{join} if we can partition the set of vertices into two or more nonempty sets in such a way that every vertex from one set is connected by an edge to every vertex in the other sets.  Also, the \emph{link} of a set $A$ of vertices in a graph is the subgraph spanned by the set of vertices that are not in $A$ and are connected by an edge to each vertex in $A$.

\begin{lemma}
\label{lemma:joins}
Let $S = S_g$ with $g \geq 2$, and let $\alpha \in \Aut \FC(S)$.  Then $\alpha$ preserves the set of separating curves in $\FC(S)$ and also preserves the set of separating multicurves in $\FC(S)$.  Moreover, $\alpha$ preserves the sides of a separating multicurve, that is, $a$ and $b$ lie on the same side of $m$ if and only if $\alpha(a)$ and $\alpha(b)$ lie on the same side of $\alpha(m)$.
\end{lemma}

\begin{proof}

First of all, it follows from the definition of $\FC(S)$ that $\alpha$ preserves multicurves and also preserves curves.  Therefore, for the first statement it suffices to distinguish the separating multicurves from the nonseparating ones.

We claim that a multicurve $m=\{c_1,\dots,c_k\}$ is separating if and only if the link of $m$ is a join.  Indeed if $m$ is a separating multicurve then the sets for the join decomposition are the curves that lie in the various complementary components of $m$ (each of these sets is nonempty because they contain curves parallel to the $c_i$).  For the other direction, we observe that if $m$ is a nonseparating multicurve, and $a$ and $b$ lie in the link of $m$, then there is a curve $d$ that intersects both $a$ and $b$.  It follows from this that the link of $m$ cannot be a join, as desired.

It follows from the argument in the previous paragraph that the two sets used to define the join decomposition for a separating multicurve are uniquely defined.  From this the second statement follows.
\end{proof}

\p{Hulls} We define the \emph{hull} of a collection of curves in a surface to be the union of the curves along with any embedded disks bounded by the curves. 

\begin{lemma}
\label{lem:hulls}
Let $S=S_g$ with $g \geq 2$, and let $\alpha \in \Aut \FC(S)$.  If $A$ is a finite set of vertices of $\FC(S)$ and a vertex $d$ lies in the hull of $A$, then $\alpha(d)$ lies in the hull of $\alpha(A)$.
\end{lemma}

\begin{proof}

It suffices to prove that $d$ lies in the hull of $A$ if and only if the link of $d$ contains the link of $A$.  To this end, suppose that $d$ is a vertex of $\FC(S)$ that does not lie in the hull of $A$.  This means that there is a component of $d \setminus A$ that lies in a component $R$ of $S \setminus A$ that is not a disk.  Because $R$ is not a disk, it contains simple closed curves that are essential in $S$, and in particular it contains one that intersects the arc of $d$ in $R$, as desired.

Suppose on the other hand that $d$ is a vertex of $\FC(S)$ that lies in the hull of $A$.  Suppose also that $e$ is a simple closed curve in $S$ that intersects $d$ but not $A$.  Since $e$ is disjoint from $A$ it must lie in the complement of $A$.  And since $d$ is contained in the hull of $A$ it must then be that $e$ lies in one of the components of $S \setminus A$ that is a disk.  It follows that $e$ is inessential, and the lemma is proven.
\end{proof}

The statement of Lemma~\ref{lem:hulls} is specifically geared towards closed surfaces.  For surfaces with punctures we would, among other things, need to define the hull to include all once-punctured disks.  

\p{Torus pairs} We now prove the main result of this subsection.  

\begin{lemma}
\label{lem:torus}
Let $S = S_g$ with $g \geq 2$, and let $\alpha \in \Aut \FC(S)$.  Then $\alpha$ preserves the set of torus pairs, the set of degenerate torus pairs, the set of nondegenerate torus pairs, and the set of torus triples.  
\end{lemma}

\begin{proof}

We proceed in four steps.  First we show that $\alpha$ preserves the union of the torus pairs and the pants pairs.  Then we show that $\alpha$ preserves the set of torus pairs.  Next, we show that $\alpha$ preserves the set of degenerate torus pairs, hence it also preserves the nondegenerate torus pairs.  Finally, we prove that $\alpha$ preserves the set of torus triples.

\medskip

\noindent \emph{Step 1.} For the first step, it suffices to show that the following three statements are equivalent for a pair of intersecting vertices $\{c,d\}$ of $\FC(S)$:
\begin{enumerate}
\item The pair $\{c,d\}$ is a torus pair or a pants pair.
\item There is at most one other vertex of $\FC(S)$ in the hull of $\{c,d\}$.
\item There is at most one other vertex of $\FC(S)$ whose link contains the link of $\{c,d\}$.   
\end{enumerate}
The second and third statements are equivalent by Lemma~\ref{lem:hulls}, so it suffices to prove the equivalence of the first two statements.  

The first statement implies the second because if $\{c,d\}$ is a torus pair or pants pair, then the hull of $\{c,d\}$ is $c \cup d$, and in this case there are either no other simple curves or one other simple curve in $c \cup d$, depending on whether or not $\{c,d\}$ is degenerate.

We prove that the second statement implies the first by considering two cases, according to whether or not $c$ and $d$ bound any disks (meaning that some complementary region of $c\cup d$ is a disk).  If $c$ and $d$ bound a disk, then there are infinitely many other vertices of $\FC(S_g)$ that lie in the hull of $\{c,d\}$; indeed, we choose any curve that agrees with with $c$ away from this disk, and disagrees with $c$ inside the given disk.  

Assume now that $c$ and $d$ bound no disks.  If $\{c,d\}$ is not a torus pair or a pants pair, then it must be that $c \cap d$ has more than one connected component.  Let $a_1$ and $a_2$ be two such components.  Let $c_1$, $c_2$, $d_1$, and $d_2$ be the closures of the complementary components of $a_1 \cup a_2$ in $c$ and $d$.  In $c \cup d$ there are four distinct simple closed curves $e_1,\dots,e_4$ that contain $c_1 \cup d_1$, $c_1 \cup d_2$, $c_2 \cup d_1$, and $c_2 \cup d_2$, respectively.  These curves intersect each $a_i$ in the empty set, an endpoint, or all of $a_i$.  The $e_i$ are all distinct from $c$ and $d$.  If some $e_i$ were inessential, then it would be the boundary of a disk in $S$.  It would follow that $c \cup d$ bounds a (possibly smaller) disk, a contradiction.

\medskip

\noindent \emph{Step 2.} For the second step, we assume that $\{c,d\}$ is a torus pair or pants pair.  We will show that, under this assumption, the pair $\{c,d\}$ is a torus pair if and only if there is a separating curve $e$ disjoint from $c$ and $d$ and with the following property: all nonseparating simple closed curves in $S_g$ lying on the same side of $e$ as $\{c,d\}$ fail to be disjoint from $c \cup d$.  The proposition then follows from the definition of $\FC(S_g)$ and Lemma~\ref{lemma:joins}.

We begin with the forward direction.  Let $\{c,d\}$ be a torus pair, let $R$ be a neighborhood of $c \cup d$ homeomorphic to a torus with one boundary component $e$.  The surface obtained by cutting $R$ along $c \cup d$ is an annulus.  Any nonseparating curve in $S_g$ that lies in $R$ is not parallel to the boundary (otherwise it is parallel to $e$, hence separating), and hence this nonseparating curve intersects either $c$ or $d$, as desired.

For the reverse direction, we assume that $\{c,d\}$ is a pants pair, and we let $e$ be any separating curve disjoint from $c \cup d$.  There is a closed neighborhood of $c \cup d$ that is a pair of pants, and where each of the three components of the boundary of this subsurface is a nonseparating curve in $S_g$ that is disjoint from $c \cup d$ and lies on the same side of $e$ as $\{c,d\}$.  This completes the proof of the second step.

\medskip

\noindent \emph{Step 3.} The following three statements are equivalent for a torus pair $\{c,d\}$ of $\FC(S)$:
\begin{enumerate}
\item The torus pair $\{c,d\}$ is nondegenerate.
\item There is exactly one other vertex of $\FC(S)$ in the hull of $\{c,d\}$.
\item There is exactly one other vertex of $\FC(S)$ whose link contains the link of $\{c,d\}$.   
\end{enumerate}
The equivalence of the first two statements can be proved by inspection of the two possible configurations for a torus pair (degenerate and nondegenerate).  The last two statements are equivalent by Lemma~\ref{lem:hulls}.  This  completes the third step.

\medskip

\noindent \emph{Step 4.} Suppose $\{c,d,e\}$ is a torus triple.  Then $e$ is the unique vertex (other than $c$ and $d$) contained in the hull of $\{c,d\}$.  By Lemma~\ref{lem:hulls}, $e$ is the unique curve whose link contains the link of $\{c,d\}$.  Since torus pairs are preserved, it now follows that torus triples are preserved.
\end{proof}

%%%
%%%
%%%

\subsection{Bigon pairs}
\label{sec:curve pairs}

In this subsection we prove Proposition~\ref{prop:curve pairs}.  We begin by defining annulus sets and describing their basic properties.  We prove in Lemma~\ref{lem:annulus} that these properties are preserved under automorphisms of $\FC(S)$.  With that in hand, we proceed to the proof of Proposition~\ref{prop:curve pairs}.

\p{Annulus sets} Suppose that $(a,b)$ is an ordered pair of disjoint, homotopic vertices of $\FC(S_g)$.  Assuming $g \geq 2$, there is a unique annulus $A$ in $S_g$ whose boundary is $a \cup b$.  Let $\FC(a,b)$ be the set of vertices of $\FC(S_g)$ contained in $A$.  We refer to $\FC(a,b)$ as an \emph{annulus set}.  We say that a pair of vertices of $\FC(S_g)$ is an \emph{annulus pair} if they lie in some $\FC(A)$.  A \emph{nonseparating noncrossing annulus pair} is an annulus pair where both curves are nonseparating and the pair is noncrossing.

There is a natural ordering on the annulus set $\FC(a,b)$: we say that $c \preceq d$ if $c$ and $d$ are noncrossing and each component of $c \setminus d$ lies in a component of $A$ bounded by $a$.

\begin{lemma}
\label{lem:annulus}
Let $S = S_g$ with $g \geq 2$, let $\alpha$ be an automorphism of $\FC(S_g)$, let $a$ and $b$ be disjoint, homotopic nonseparating curves.
\begin{enumerate}[itemsep=.15em]
\item\label{item:pairs} The curves $\alpha(a)$ and $\alpha(b)$ are disjoint, homotopic nonseparating curves.
\item\label{item:annuli} The image of $\FC(a,b)$ under $\alpha$ is $\FC(\alpha(a),\alpha(b))$.
\item\label{item:noncross} If $c,d \in \FC(a,b)$ are noncrossing then $\alpha(c)$ and $\alpha(d)$ are noncrossing.
\item\label{item:order} If $c \preceq d$ in $\FC(a,b)$ then $\alpha(c) \preceq \alpha(d)$ in $\FC(\alpha(a),\alpha(b))$.
\end{enumerate}
\end{lemma}

\begin{proof}

We prove the four statements in turn.  The first statement is a consequence of Lemma~\ref{lemma:joins} and the fact that 
two disjoint nonseparating curves $a$ and $b$ in $S_g$ are homotopic if and only if the following conditions hold: $a$ and $b$ form a separating multicurve and all separating curves disjoint from both $a$ and $b$ lie on the same side of the multicurve $a \cup b$.

The second statement is an immediate consequence of the first statement and Lemmas~\ref{lemma:joins}.

We proceed to the third statement.  By the first statement, we may assume that $\alpha$ preserves $a$ and $b$ (that is, we may postcompose $\alpha$ with an automorphism induced by an element of $\Homeo(S)$ to make this so).  By the second statement, $\alpha$ preserves the annulus set $\FC(a,b)$.  Two curves $c$ and $d$ are noncrossing if and only if there is a different curve $e \in \FC(a,b)$ with the property that every curve in $\FC(a,b)$ that intersects $c$ and $d$ must also intersect $e$ (when $c$ and $d$ are noncrossing this curve $e$ contains $c \cap d$ and passes through the interior of each bigon formed by $a$ and $b$).  The third statement follows now from the previous two.  

The fourth statement holds by the previous three statements and the fact that noncrossing curves $c,d \in \FC(a,b)$ satisfy $c \preceq d$ if and only if there is an element of $\FC(a,b)$ that intersects $a$ and $c$ but not $d$.  This completes the proof of the lemma.
\end{proof}

\p{Type 1 and type 2 curves} Suppose that $\{c,d\}$ is a nonseparating noncrossing annulus pair, and suppose that $e$ is a curve so that $\{c,e\}$ and $\{d,e\}$ are degenerate torus pairs.  If $c \cap e$ and $d \cap e$ are the same point, then we say that $e$ is a \emph{type 1 curve} for $\{c,d\}$.  Otherwise we say that $e$ is a \emph{type 2 curve} for $\{c,d\}$.

\begin{lemma}
\label{lem:type12}
Let $S = S_g$ with $g \geq 2$, and let $\alpha \in \Aut \FC(S)$.  Then $\alpha$ preserves type 1 and type 2 curves for nonseparating noncrossing annulus pairs.  More precisely, if $\{c,d\}$ is a nonseparating noncrossing annulus pair and $e$ is a type 1 curve for $\{c,d\}$, then $\alpha(e)$ is a type 1 curve for the nonseparating noncrossing annulus pair $\{\alpha(c),\alpha(d)\}$, and similarly for type 2 curves.
\end{lemma}

\begin{proof}

Since $\alpha$ preserves degenerate torus pairs (Lemma~\ref{lem:torus}), we may assume that $\alpha$ preserves the union of the type 1 and type 2 curves for $\{c,d\}$.  So it remains to show that $\alpha$ preserves the two types.

Say that $c,d \in \FC(a,b)$.  By parts \eqref{item:pairs} and \eqref{item:annuli} of Lemma~\ref{lem:annulus}, we may assume without loss of generality that $\alpha$ preserves $\FC(a,b)$ (as in the proof of Lemma~\ref{lem:annulus} we may postcompose $\alpha$ with an automorphism induced by an element of $\Homeo(S)$ to make this so).

Let $e$ be a curve with the property that $\{c,e\}$ and $\{d,e\}$ are degenerate torus pairs, so $e$ is either a type 1 or type 2 curve for $\{c,d\}$.  We claim that $e$ is a type 2 curve if and only if there is a curve $f$ with the following properties:
\begin{itemize}[itemsep=.15em]
\item $f$ is contained in the the hull of $\{c,d,e\}$,
\item $f$ is not contained in $\FC(a,b)$, and
\item $f$ is not equal to $e$.
\end{itemize}
The forward direction is proved by construction, as follows.  If $e$ is a type 2 curve, it passes through the interior of a bigon $B$ bounded by arcs of $c$ and $d$.  By replacing the arc of $e$ that passes through the bigon with a different arc, we obtain the desired curve $f$.  For the other direction, we suppose that $e$ is a type 1 curve for $\{c,d\}$.  Any curve that satisfies the first two given properties would have to contain all of $e$, hence would fail the third property.  The claim follows and the lemma thus follows from Lemma~\ref{lem:hulls}.
\end{proof}

\p{Bigon pairs} We are now ready for the proof of Proposition~\ref{prop:curve pairs}.

\begin{proof}[Proof of Proposition~\ref{prop:curve pairs}]

By Lemma~\ref{lem:annulus}\eqref{item:noncross}, we have that $\alpha$ preserves nonseparating noncrossing annulus pairs.  If $\{c,d\}$ is such a pair, then $c \cup d$ is a union of inessential curves with (possibly degenerate) arcs connecting these inessential curves cyclically.  If there are $k$ (or more) inessential curves, then we may find curves $e_1,f_1,\dots,e_k,f_k$ with the following properties:
\begin{itemize}
\item any two of the $2k$ curves are homotopic and pairwise disjoint,
\item each $e_i$ is a type 1 curve for $\{c,d\}$,
\item each $f_i$ is a type 2 curve for $\{c,d\}$,
\item and the $2k$ curves lie in the given order in the annulus bounded by $e_1$ and $f_k$.  
\end{itemize}
Conversely, if we can find $2k$ curves with the above properties, then $c$ and $d$ form $k$ inessential curves.  Thus, by Lemmas~\ref{lem:annulus} and~\ref{lem:type12}, $\alpha$ preserves the number of inessential curves formed by $\{c,d\}$.  In particular, it preserves the set of noncrossing annulus pairs that form exactly one inessential curve.

A bigon pair is a nonseparating noncrossing annulus pair $\{c,d\}$ that forms exactly one inessential curve and has the additional property that $c \cap d$ is a nondegenerate interval.  Among the nonseparating noncrossing annulus pairs forming exactly one inessential curve, the bigon pairs are exactly those for which there exists two curves $e_1$ and $e_2$ that are disjoint and are both type 1 curves for $\{c,d\}$.  The proposition follows.
\end{proof}

%%%
%%%
%%%

\subsection{Sharing pairs}
\label{sec:curve pair pairs}

The goal of this subsection is to prove Proposition~\ref{prop:curve pair pairs}, which states that automorphisms of $\FC(S_g)$ preserve sharing pairs.  

In the following proof we write $A \doteq B$ if $A$ and $B$ are two sets with $A \triangle B$ is a finite set.  The relation $\doteq$ is an equivalence relation. 

If $\{c,d,e\}$ is a torus triple, then we have, for example, that $e \doteq c \triangle d$.  Similarly, if $\{c,d\}$ is a bigon pair, then the inessential curve $e$ determined by $\{c,d\}$ satisfies $e  \doteq c \triangle d$.  In what follows we will use the fact that if $e$ and $e'$ are curves in a surface with $e \doteq e'$ then $e=e'$.  

\begin{proof}[Proof of Proposition~\ref{prop:curve pair pairs}]

We claim that bigon pairs $\{c,d\}$ and $\{c',d'\}$ form a sharing pair if and only if the following conditions hold:
\begin{enumerate}
\item each of $\{c,d'\}$ and $\{c',d\}$ is an nondegenerate torus pair, and
\item there is a curve that forms a torus triple with both $\{c,d'\}$ and $\{c',d\}$.
\end{enumerate}
The proposition follows from the claim, Proposition~\ref{prop:curve pairs}, and Lemma~\ref{lem:torus}.  The first direction of the claim can be verified from the picture.  In Figure~\ref{fig:determined} we indicate the curve $e$ that forms a torus triple with both $\{c,d'\}$ and $\{c',d\}$.

\begin{figure}
\includegraphics[scale=.275]{sharing_pair}
\hspace{.15in}
\includegraphics[scale=.275]{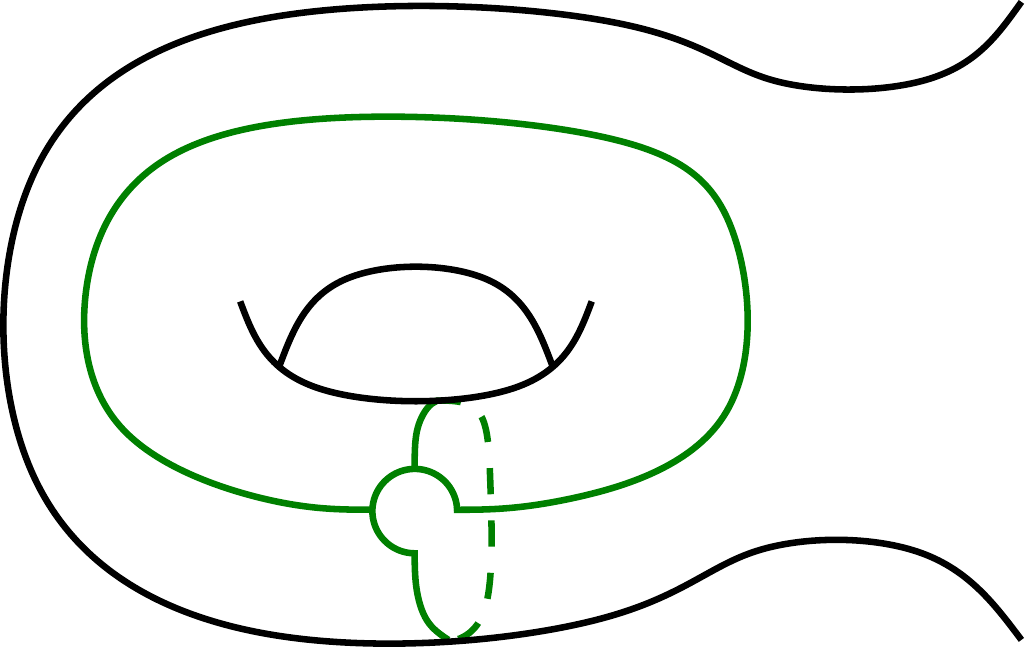}
\hspace{.15in}
\includegraphics[scale=.275]{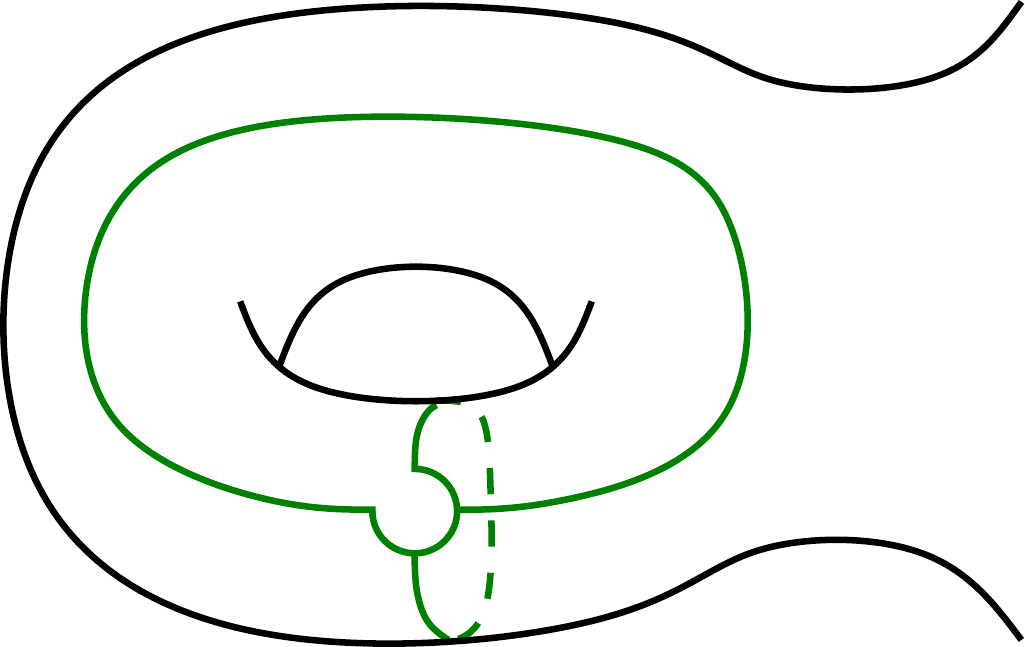}
\hspace{.15in}
\includegraphics[scale=.275]{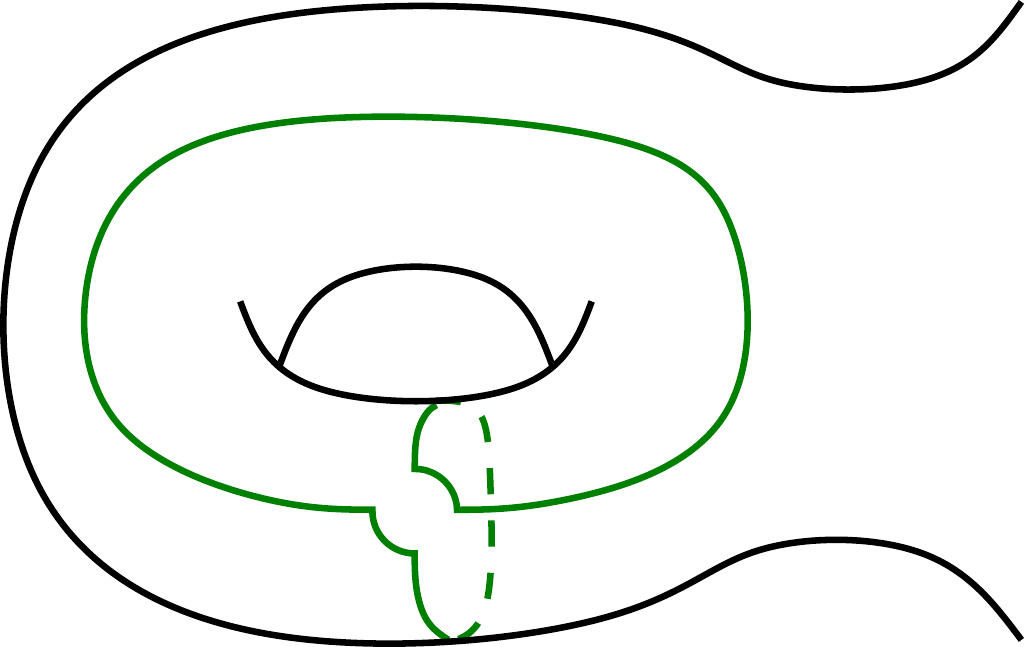}
\caption{\emph{Left to right:} a sharing pair $\{\{c,d\},\{c',d'\}\}$, the torus pair $\{c',d\}$, the torus pair $\{c,d'\}$, and the unique curve $e$ that forms a torus triple with both $\{c',d\}$ and $\{c,d'\}$}
\label{fig:determined}
\end{figure}

For the other direction, say that $e$ is the inessential curve determined by $\{c,d\}$, and that $e'$ is the inessential curve determined by $\{c',d'\}$.  Since there is a curve that forms a torus triple with both $\{c,d'\}$ and $\{c',d\}$, this means that $c \triangle d' \doteq c' \triangle d$.  Using the basic fact about symmetric differences that $A \triangle B = (A \triangle C) \triangle (B \triangle C)$, we have: 
\[
e \doteq c \triangle d = \left(c \triangle c'\right) \triangle \left(c' \triangle d\right) \doteq \left(c \triangle c'\right) \triangle \left(c \triangle d'\right) = c' \triangle d' \doteq e'.
\]
Thus $e=e'$, which is to say that $\{c,d\}$ and $\{c',d'\}$ determine the same inessential curve $e$.  We identify the complement of the interior of $e$ with $S_g^1$.  The pairs $\{c,d\}$ and $\{c',d'\}$ determine arcs $a$ and $a'$ in $S_g^1$.  If $a$ and $a'$ were not disjoint and linked, then this would violate the condition that $\{c,d'\}$ (and also $\{c',d\}$) is a nondegenerate torus pair.  This completes the proof. 
\end{proof}

%%%
%%%
%%%

\section{Connectedness of fine arc graphs}
\label{sec:fa}

Let $S_g^b$ denote the surface obtained from $S_g$ by deleting the interiors of $b$ disjoint disks.  Let $S = S_g^b$ with $b > 0$.  The goal of this section is to prove that three fine arcs graphs are connected: the fine arc graph $\FA(S)$, the fine nonseparating arc graph $\NA(S)$, and the fine linked arc graph $\LkA(S)$.  We begin by proving that $\FA(S)$ is connected (Proposition~\ref{prop:fa}), and then derive the connectivity of the other two graphs as corollaries (Corollaries~\ref{cor:nfa} and~\ref{cor:lka}).  For the proof of Theorem~\ref{thm:main}, we will only use the connectivity of $\LkA(S)$.

\p{The fine arc graph} We begin with the basic definitions.  An  \emph{arc} in $S=S_g^b$ is the image of a map $a: [0,1] \to S$.  We say that the arc is \emph{simple} if the map $a$ is injective, we say that the arc is \emph{proper} if $a^{-1}(\partial S)=\{0,1\}$, and we say that the arc is \emph{essential} if it is not homotopic into $\partial S$.  We say that two arcs have \emph{disjoint interiors} if they are disjoint away from $\partial S$.  When two arcs have no intersections at all (including at the boundary), we say that the arcs are \emph{completely disjoint}.

The fine arc graph $\FA(S)$ is the graph whose vertices are essential simple proper arcs in $S$ and whose edges connect vertices with disjoint interiors.  The next proposition states that $\FA(S)$ is connected; we note that for $S=S_0^1$, the graph $\FA(S)$ is empty, hence vacuously connected.

\begin{proposition}
\label{prop:fa}
For any $S=S_g^b$ with $b > 0$, the graph $\FA(S)$ is connected.
\end{proposition}

Our proof of Proposition~\ref{prop:fa} is based on the proof of Bowden--Hensel--Webb that the fine (smooth) curve graph is connected \cite[Section 3]{BHW}.  As that proof relies on the connectivity of the classical curve graph, our proof relies in the connectivity of the classical arc graph $\A(S)$.  The vertices of $\A(S)$ are isotopy classes of essential simple proper arcs in $S$, where isotopies are allowed to move endpoints of arcs along the boundary of $S$.  The edges are pairs of vertices with disjoint representatives.  The arc complex is the flag complex associated to $\A(S)$.

\begin{proof}[Proof of Proposition~\ref{prop:fa}]

There is a natural simplicial map $\FA(S) \to \A(S)$ given by taking isotopy classes.  The arc complex is contractible \cite{Hatcher}, so in particular its 1-skeleton $\A(S)$ is connected.  Thus, it suffices to show that for any vertex of $\A(S)$, the subgraph of $\FA(S)$ spanned by its preimage is connected.  In other words, it suffices to show that between any two isotopic essential simple proper arcs in $S$ there is a path in $\FA(S)$ connecting the two.

Let $a$ and $b$ be isotopic vertices of $\A(S)$ and let $H : S^1 \times [0,1] \to S$ be an isotopy from $a$ to $b$.  For $t \in [0,1]$, let $a_t$ be the image of $S^1 \times \{t\}$, so $a_0=a$ and $a_1=b$.  For each vertex $c$ of $\FA(S)$, we define
\[
I_c = \{ t \in [0,1] \mid a_t \text{ is completely disjoint from } c \}
\]
Each $I_c$ is open.  Thus we may find a sequence of open intervals $I_0, \dots, I_k$ that cover $[0,1]$ and so each $I_i$ is some $I_{c_i}$.  We may further assume that $0 \in I_0$ and that $I_i \cap I_{j}$ is nonempty if and only if $|i-j|=1$.  Let $t_0=0$, let $t_{k+1}=1$, and for $i \in \{1,\dots,k\}$ let $t_i$ be an element of $I_i \cap I_{i+1}$.  For $i \in \{0,\dots,k+1\}$ let $a_i = a_{t_i}$.

By definition, each pair $\{a_i,a_{i+1}\}$ is completely disjoint from $c_i$.  Thus the sequence
\[
a=a_0,c_0,a_1,c_1,a_2,\dots,a_{k-1},c_{k-1},a_k=b
\]
is the desired path from $a$ to $b$ in $\FA(S)$.
\end{proof}

We remark that it is possible to define $\FA(S)$ in a similar manner when $S$ has punctures instead of boundary.  However, in some cases, this graph is not connected.  For instance, if $S=S_{g,1}$ and an arc $a$ spirals around the puncture infinitely many times relative to $b$, then $a$ and $b$ are not connected by a path in $\FA(S)$.  The part of the proof of Proposition~\ref{prop:fa} that fails for surfaces with punctures is that the sets $I_c$ are not always open.

\p{The fine nonseparating arc graph} We say that an arc in a surface $S$ is \emph{nonseparating} if its complement in $S$ is connected.  The fine nonseparating arc graph is the subgraph of $\FA(S)$ spanned by the nonseparating arcs.  

\begin{corollary}
\label{cor:nfa}
For any $S=S_g^b$ with $b > 0$, the graph $\NA(S)$ is connected.
\end{corollary}

\begin{proof}

Let $a$ and $b$ be vertices of $\NA(S)$. By Proposition~\ref{prop:fa} there is a path
\[
a=a_0,\dots,a_k=b
\]
in $\FA(S)$ between $a$ and $b$.  We will show that we can replace this path with one that lies entirely in $\NA(S)$, by either removing the separating vertices or replacing them with nonseparating vertices.   

To this end we have the following claim: if $x$ is a separating arc in $S$ and $R \subseteq S$ is one side of $x$ (meaning the closure of one of the two complementary components), then there is a nonseparating arc $y$ in $S$ that is contained in $R$.  Since $x$ is separating it must be that the endpoints of $x$ lie on the same component of $\partial S$, say $d$.  If $R$ is a planar surface then since $x$ is essential $R$ must contain some other component of $\partial S$, say $d'$.  As such we may take $y$ to be any arc in $R$ connecting $d$ to $d'$.  If $R$ has positive genus, then we may take $y$ to be any nonseparating arc in $R$ connecting $d$ to itself.  

We may now complete the proof of the lemma.  Suppose that the given path from $a$ to $b$ has a separating arc $a_i$.  If $a_{i-1}$ and $a_{i+1}$ have disjoint interiors, we may remove $a_i$ from the path.  Otherwise, it must be that $a_{i-1}$ and $a_{i+1}$ lie on the same side of $a_i$.  By the claim there is a nonseparating arc $a_i'$ that is contained on the other side of $a_i$.  This $a_i'$ must be connected by edges to both $a_{i-1}$ and $a_{i+1}$.  Thus we may replace $a_i$ with $a_i'$ in the given path.  The lemma follows.
\end{proof}

\p{The fine linked arc graph} Let $S=S_g^b$ be a surface with $g \geq 1$ and $b > 0$.  Let $d_0$ be a distinguished component of $\partial S$. We say that two vertices $a$ and $b$ of $\FA(S)$ with disjoint interiors are \emph{linked} at $d_0$ if all four endpoints lie on $d_0$ and the boundary curve for any sufficiently small neighborhood of $d_0$ alternates between intersections with $a$ and $b$.  Some examples of linked arcs in $S_g^1$ with disjoint interiors are shown in Figure~\ref{fig:linked}.

\begin{figure}
\includegraphics[scale=.25]{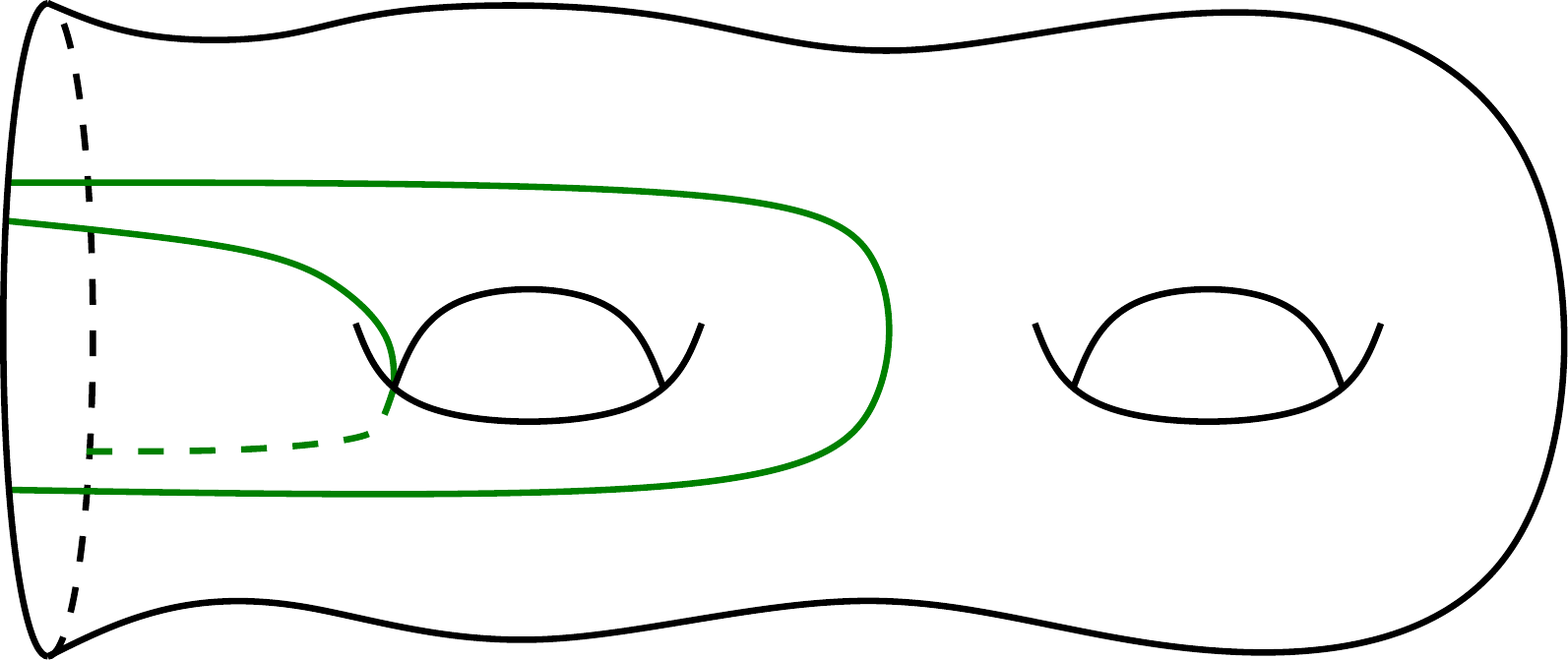}
\hspace{.5in}
\includegraphics[scale=.25]{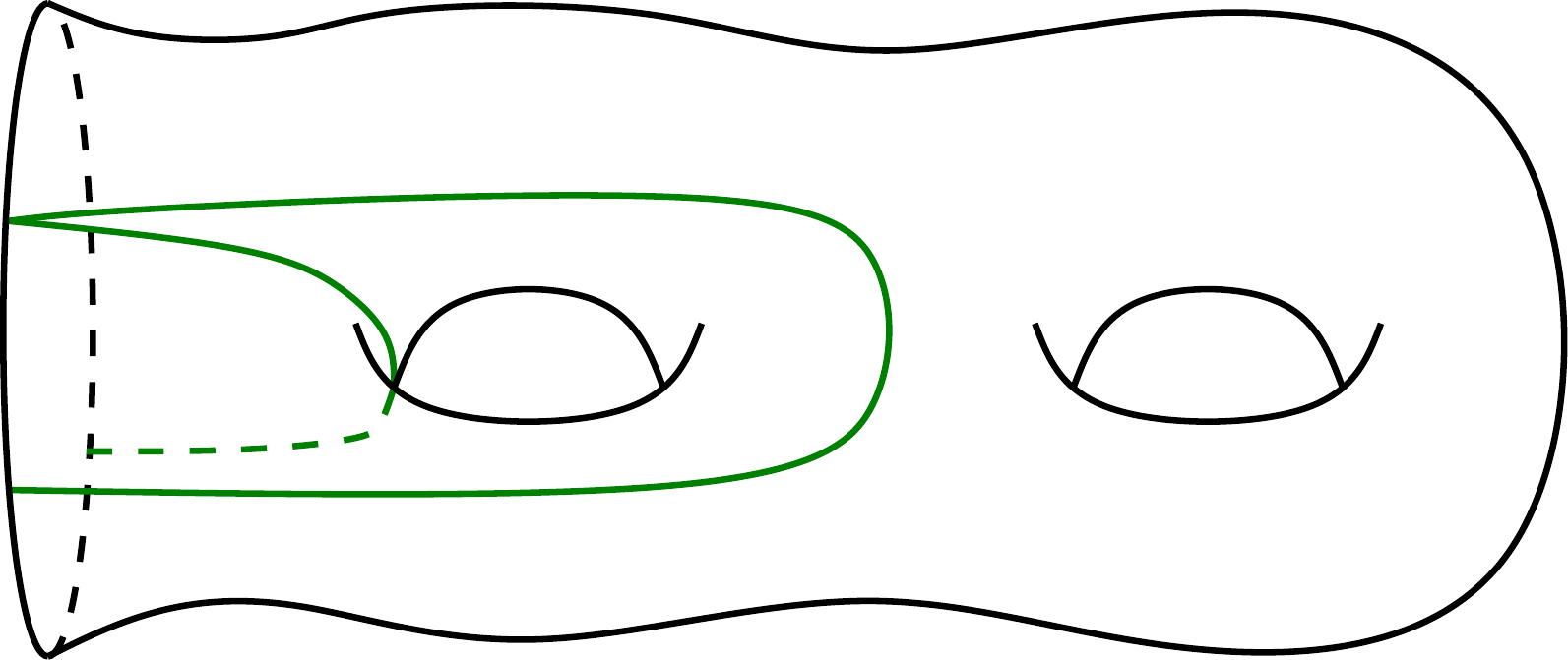}
\caption{Two types of linked pairs of arcs in $S_2^1$}
\label{fig:linked}
\end{figure}

We define $\LkA(S,d_0)$ to be the graph whose vertices are nonseparating simple proper arcs in $S$ with both endpoints at $d_0$ and whose edges connect arcs with disjoint interiors that are linked at $d_0$.  When convenient we suppress $d_0$ in the notation in what follows and write $\LkA(S)$.

\begin{corollary}
\label{cor:lka}
Let $S=S_g^b$ with $g \geq 1$ and $b > 0$, and let $d_0$ be a component of $\partial S$.  The graph $\LkA(S,d_0)$ is connected.
\end{corollary}

\begin{proof}

Let $a$ and $b$ be vertices of $\LkA(S,d_0)$. By Corollary~\ref{cor:nfa} there is a path
\[
a=a_0,\dots,a_k=b
\]
in $\NA(S)$.   

For a given edge $\{a_i,a_{i+1}\}$ in this path where $a_i$ and $a_{i+1}$ are not linked, we would like to show there is an arc $b_i$ that is linked with both $a_i$ and $a_{i+1}$ and disjoint from their interiors.  For then we may obtain a path from $a$ to $b$ in $\LkA(S)$ by inserting all such $b_i$ into the above sequence of vertices.

So let $\{x,y\}$ be an arbitrary edge in $\NA(S)$ where $x$ and $y$ are unlinked.  Since each of $x$ and $y$ is nonseparating, it follows that $x \cup y$ separates $S$ into at most two components.  We consider a small annular neighborhood $A$ of $d_0$.  The intersections of $x$ and $y$ with $A$ divide it into 4 components.  These components come in a cyclic order; call them $A_0$, $A_1$, $A_2$, and $A_3$.  Exactly two of the 4 components of $A$ have the property that they are bounded by one arc of $x$ and one arc of $y$.  These components are not adjacent; say they are $A_0$ and $A_2$.  Since $x \cup y$ separates $S$ into at most two components, $A_1$ and $A_3$ lie in the same component of $S$ cut along $x \cup y$.  Thus there is an arc $z$ in $S$ that is disjoint from $x$ and $y$ away from $d_0$ and connects $A_1$ to $A_3$.  This arc $z$ is thus linked with both $x$ and $y$ by definition.  By virtue of being linked with other arcs, $z$ is nonseparating, hence a vertex of $\LkA(S,d_0)$.  The corollary follows.
\end{proof}

%%%
%%%
%%%

\section{Automorphisms of the extended fine curve graph}
\label{sec:efc}

In this section we prove Theorem~\ref{thm:efc}, which states that the natural map $\nu : \Homeo(S) \to \EFC(S)$ is an isomorphism.  As discussed in the introduction, the proof we give is the original one, due to Farb and the second author.  We emphasize that the proof applies to all surfaces without boundary, including those whose fundamental group is not finitely generated.  

In Section~\ref{sec:conv} we introduce convergent sequences of vertices and prove several related lemmas about them.  Then in Section~\ref{sec:efc pf} we use convergent sequences to prove Theorem~\ref{thm:efc}.

\subsection{Convergent sequences}
\label{sec:conv}

Let $S$ be a surface without boundary.  We say that a sequence of vertices $(c_i)$ of $\EFC(S)$ \emph{converges} to a point $x \in S$ if the corresponding curves converge to $x$ in the usual sense: every neighborhood of $x$ contains all but finitely many of the $c_i$.  In this case we write $\lim(c_i) = x$.  If $(c_i)$ is convergent, it must be that there exists $M > 0$ so that each $c_i$ with $i > M$ is inessential.

The main goal of this section is to prove Lemma~\ref{lemma:conv}, which states that automorphisms of $\EFC(S)$ preserve convergent sequences.  We also prove several related statements, Corollaries~\ref{cor:coincident}, \ref{cor:conv conv}, and~\ref{cor:conv curve}. Before proving these we introduce a technical tool, the connect-the-dots lemma, Lemma~\ref{lemma:connect}.  

\p{Connect-the-dots lemma} We will use the following technical lemma to prove that automorphisms of $\EFC(S)$ preserve convergent sequences.

\begin{lemma}
\label{lemma:connect}
Let $S$ be a surface without boundary, let $(c_i)$ be an infinite sequence of vertices of $\EFC(S)$, and suppose $R$ is a compact subsurface of $S$ with the property that infinitely many of the $c_i$ intersect $R$.  Then there is a simple closed curve $a$ in $R$ that intersects infinitely many of the $c_i$.  
\end{lemma}

\begin{proof}

In the case where infinitely many $c_i$ intersect the boundary $\partial R$, we may take $a$ to be any component of $\partial R$ that intersects infinitely many $c_i$.  The other case is where infinitely many $c_i$ intersect the interior of $R$.  For each $i$ we choose a point $x_i$ in $c_i \cap R$.  Up to passing to a subsequence, we may assume that the $x_i$ converge to a point $x \in R$.  We may also assume without loss of generality that $x$ lies in the interior or $R$.  Indeed each component of each $c_i \setminus \partial R$ is an open interval, and so we may choose each $x_i$ to lie outside of some fixed neighborhood of $\partial R$.  

Let $U$ be an open disk in $R$ that contains $x$.  By the classification of surfaces of infinite type \cite[Theorem 1]{Richards}, the surface $U \setminus \left( \{x_i\} \cup \{x\} \right)$ is homeomorphic to $F = \C \setminus \left( \{1/n\} \cup \{0\} \right)$ (this surface is sometimes called the flute surface).  Regarding the punctures in $F$ as marked points, there is clearly a simple closed curve containing infinitely many marked points of $F$ (any curve containing a nontrivial interval $[0,\epsilon]$).  Any such curve corresponds to a curve in $U$ containing infinitely many $x_i$.  This completes the proof.
\end{proof}

\p{Convergent sequences are preserved} For the statement of Lemma~\ref{lemma:conv}, we require one more definition.  We say that a vertex $a$ of $\EFC(S)$ \emph{intersects the tail} of a convergent sequence of vertices $(c_i)$ if there are infinitely many $i$ so that $a$ intersects $c_i$.  

\begin{lemma}
\label{lemma:conv}
Let $S$ be a surface without boundary.  Automorphisms of $\EFC(S)$ preserve convergent sequences.
\end{lemma}

\begin{proof}

It suffices to prove the following statement.  Suppose that $(c_i)$ is a sequence of vertices of $\EFC(S)$.  Then $(c_i)$ is convergent if and only if the following two conditions hold:
\begin{enumerate}
\item there exists a vertex $a$ that intersects the tail of $(c_i)$, and
\item if $a$ and $b$ are distinct vertices that intersect the tail of $(c_i)$, then $a$ and $b$ intersect.
\end{enumerate}

The forward direction follows immediately from the definition of a convergent sequence and the fact that vertices of $\EFC(S)$ correspond to closed subsets of $S$.  Indeed, if $(c_i)$ converges to the point $x$ then any two vertices $a$ and $b$ that intersect the tail of $(c_i)$ must intersect at the point $x$.

For the reverse direction, suppose that $(c_i)$ is not convergent.  In the case where the sequence $(c_i)$ leaves every compact subsurface of $S$, the first condition fails: there is no vertex $a$ intersecting each $c_i$.  Thus, we may assume that there is a compact subsurface $R \subseteq S$ with the property that $c_i \cap R$ is nonempty for infinitely many $i$.  We will show that there exist disjoint curves $a$ and $b$ that intersect the tail of $(c_i)$.

Since $R$ is compact, we may choose a subsequence $(c_{i_j})$ of $(c_i)$ and a sequence of points $x_{i_j} \in c_{i_j}$ with the property that $x_{i_j}$ converges to a point $x$ in $R$.  Let $U$ be a closed disk in $R$ and let $V$ be an open disk in $R$ with the following properties:
\begin{enumerate}
\item $x \in U \subseteq V \subseteq R$,
\item $R \setminus V$ is a compact surface, and
\item infinitely many $c_i$ are not contained in $V$.
\end{enumerate}
The third condition is attainable since $(c_i)$ is not convergent.  By passing to a further subsequence, we may assume that each $c_{i_j}$ lies in $U$.  By Lemma~\ref{lemma:connect}, there is a curve $a$ in $U$, hence in $S$, with the property that $a$ intersects infinitely many $c_{i_j}$.  

By the second and third conditions on $U$ and $V$ and the assumption that $(c_i)$ does not leave $R$, we have that $(c_i)$ and $R \setminus V$ satisfy the hypotheses of Lemma~\ref{lemma:connect} (the subsurface $R \setminus V$ here is the subsurface $R$ in the statement of the lemma).  Thus by Lemma~\ref{lemma:connect} there is a curve $b$ in $R \setminus V$ that intersects infinitely many of the $c_i$.  The curves $a$ and $b$ are disjoint since $a \subseteq U$ and $b \subseteq R \setminus U$.  Both curves intersect the tail of $(c_i)$ by construction.  This completes the proof.
\end{proof}

\p{Coincidence of convergent sequences is preserved} For the statement of the following corollary, we say that two convergent sequence of vertices of $\EFC(S)$ are \emph{coincident} if they converge to the same point of $S$.  We also define the \emph{interleave} of two sequences $(c_i)$ and $(d_i)$ to be the sequence $c_1,d_1,c_2,d_2,\dots$.  

We have the following corollary of Lemma~\ref{lemma:conv}.  The first statement follows from the definition of convergence in point set topology, and the second statement follows from Lemma~\ref{lemma:conv}.

\begin{corollary}
\label{cor:coincident}
Let $S$ be a surface without boundary.  Let $(c_i)$ and $(d_i)$ be two convergent sequences of vertices of $\EFC(S)$.  Then $(c_i)$ and $(d_i)$ are coincident if and only if the interleave of $(c_i)$ and $(d_i)$ is convergent.  In particular, automorphisms of $\EFC(S)$ preserve coincidence of convergent sequences.
\end{corollary}

\p{Convergence of convergent sequences is preserved} For the next corollary to Lemma~\ref{lemma:conv}, we say that a sequence of convergent sequences
\[
(c_i^1), (c_i^2), (c_i^3), \dots
\]
in $\EFC(S)$ \emph{converges} if the sequence of limit points
\[
\lim (c_i^1), \lim (c_i^2), \lim (c_i^3), \dots
\]
converges to a point $x \in S$.  In this case we say that the sequence converges to $x$.  

A \emph{diagonal sequence} for a sequence of sequences as above is a sequence $(d_j)$ with each $d_j$ equal to some $c_i^j$.  In other words, there is a function $D : \N \to \N$ so that $d_j = c_{D(j)}^j$.  We impose a partial order on diagonal sequences for convergent sequences as follows: $(d_j) \preceq (e_j)$ if the corresponding functions satisfy $D(j) \leq E(j)$ for all $j$.  In the statement of the next corollary, we say that a diagonal subsequence is sufficiently large if it is sufficiently large with respect to this ordering.

Let $(c_i^1), (c_i^2), (c_i^3), \dots$ be a sequence of convergent sequences of vertices of $\EFC(S)$ and let $x \in S$.  Then this sequence converges to $x \in S$ if and only if all sufficiently large diagonal subsequences converge to $x$.  We thus have the following consequence of Lemma~\ref{lemma:conv}.

\begin{corollary}
\label{cor:conv conv}
Let $S$ be a surface without boundary.  Automorphisms of $\EFC(S)$ preserve convergent sequences of convergent sequences of vertices of $\EFC(S)$.  
\end{corollary}

\p{Convergence of a sequence to a curve is preserved}  For the next corollary, we say that a sequence $(c_i)$ of vertices of $\EFC(S)$ converges to a vertex $c$ if
\[
\lim (c_i) \in c.
\]
In this case we say that $c$ is a \emph{limit curve} for $(c_i)$.  We have that $c$ is a limit curve for $(c_i)$ if and only if the following condition holds: if $a$ is any vertex of $\EFC(S)$ that intersects the tail of $(c_i)$ then $a$ intersects $c$.  In particular we have the following corollary of Lemma~\ref{lemma:conv}, which follows by an argument similar to the one used for Lemma~\ref{lemma:conv}.  

\begin{corollary}
\label{cor:conv curve}
Let $S$ be a surface without boundary.  Automorphisms of $\EFC(S)$ respect the relationship between convergent sequences and limit curves.  More precisely, $c$ is a limit curve for a sequence of vertices $(c_i)$ and $\alpha$ is an element of $\Aut \EFC(S)$, then $\alpha(c)$ is a limit curve for $(\alpha(c_i))$.  
\end{corollary}

\subsection{Finishing the proof}
\label{sec:efc pf}

We require one more lemma for the proof of Theorem~\ref{thm:efc}.

\begin{lemma}
\label{lemma:inj1}
For any surface $S$ without boundary, the natural map
\[
\nu : \Homeo(S) \to \Aut \EFC(S)
\]
is injective.
\end{lemma}

\begin{proof}

Suppose that $f \in \Homeo(S)$ lies in $\ker \nu$, and let $x \in S$.  Let $c$ and $d$ be two vertices of $\EFC(S)$ with $c \cap d = \{x\}$.  Since $f(c)=c$ and $f(d)=d$, it follows that $f(x) = f(c \cap d) = c \cap d = x$.  Since $x$ was arbitrary, $f$ is the identity, as desired.
\end{proof}

\begin{proof}[Proof of Theorem~\ref{thm:efc}]

As in the statement of the theorem, let $\nu : \Homeo(S) \to \Aut \EFC(S)$ be the natural map.  As per the statement, we would like to show that $\nu$ is an isomorphism.  By Lemma~\ref{lemma:inj1}, the map $\nu$ is injective.  

We wish to construct a left inverse $\xi : \Aut \EFC(S) \to \Homeo(S)$ for $\nu$, as this will imply that $\nu$ is surjective.  For $\alpha$ an arbitrary element of $\Aut \EFC(S)$ let $f_\alpha : S \to S$ be the map given by the following rule: for $x \in S$ we choose a convergent sequence $(c_i)$ of vertices of $\EFC(S)$ with $\lim (c_i) = x$ and define
\[
f_\alpha(x) = \lim (\alpha(c_i)).
\]
The right hand side is well defined because $\alpha$ preserves convergent sequences (Lemma~\ref{lemma:conv}).  The function $f_\alpha$ is well defined and bijective by Corollary~\ref{cor:coincident}.

Our next goal is to show that each such $f_\alpha$ is a homeomorphism of $S$.  Since $f_\alpha^{-1} = f_{\alpha^{-1}}$, it suffices to show that $f_\alpha$ is continuous.  And since surfaces are first countable, the continuity of $f_\alpha$ can be verified by showing that it preserves limits points of convergent sequences.  But this is precisely the content of Corollary~\ref{cor:conv conv}.

Now that we have shown that $\xi : \Aut \EFC(S) \to \Homeo(S)$ is a well-defined homomorphism, it remains to show that $\xi \circ \nu$ is the identity.  To this end we require the following.

\medskip

\begin{itemize}[leftmargin=12.5ex]
\item[\emph{Claim 1.}] If $\alpha$ is an element of $\Aut \EFC(S)$ and $c$ is a vertex of $\EFC(S)$, then
\[
\xi(\alpha)(c) = \alpha(c).
\]
\item[\emph{Claim 2.}] If $f$ is an element of $\Homeo(S)$ and $c$ is a vertex of $\EFC(S)$, then 
\[
\nu(f)(c) = f(c).
\]
\end{itemize}

\medskip

\noindent The first claim follows from Corollary~\ref{cor:conv curve}, and the second follows from the definition of the natural map $\nu$.  

We may now prove that $\xi \circ \nu$ is the identity.  Since $\nu$ is injective, it suffices to show that $\nu \circ \xi \circ \nu(f)=\nu(f)$ for each $f \in \Homeo(S)$.  This is to say that $\xi \circ \nu(f)$ and $f$ have the same action on the set of vertices of $\EFC(S)$.  Let $c$ be an arbitrary vertex of $\EFC(S)$.  Applying the two claims in the previous paragraph in succession we have
\[
\xi \circ \nu(f)(c) = \nu(f)(c) = f(c).
\]
This completes the proof of the theorem.
\end{proof}

%%%
%%%
%%%

\section{Automorphisms of the fine curve graph}
\label{sec:pf}

For the proof of Theorem~\ref{thm:main}, we require one additional lemma.  The proof is the same as the proof of Lemma~\ref{lemma:inj1}

\begin{lemma}
\label{lemma:inj}
For $g \geq 2$, the natural map $\eta : \Homeo(S_g) \to \Aut \FC(S_g)$ is injective.
\end{lemma}

The proof of Theorem~\ref{thm:main} also requires a definition.  For a graph $\Gamma$ and a subgraph $\Delta$, we say that a map $\Aut \Delta \to \Aut \Gamma$ is an \emph{extension map} if each element of the image preserves $\Delta$ and further that each element of $\Aut \Delta$ is equal to the restriction of its image.

\begin{proof}[Proof of Theorem~\ref{thm:main}]

The proof has two steps.  The first step is to show that there exists an extension homomorphism $\varepsilon : \Aut \FC(S_g) \to \Aut \EFC(S_g)$.  The second step is to use $\varepsilon$ to complete the proof of the theorem.

\medskip

\noindent \emph{Step 1.} Let $\alpha \in \Aut \FC(S_g)$.  We would like to define an element $\hat \alpha \in \Aut \EFC(S_g)$.  We will then define $\varepsilon(\alpha)$ to be $\hat \alpha$.  For any essential simple closed curve $c$ in $S_g$ we define $\hat \alpha (c)$ to be $\alpha(c)$.  For an inessential curve $e$ in $S_g$, we take any bigon pair $\{c,d\}$ determining $e$ and define $\hat \alpha(e)$ to be the inessential curve determined by $\{\alpha(c),\alpha(d)\}$; this makes sense because of Proposition~\ref{prop:curve pairs}.

We would like to show that $\hat \alpha$ is a well defined bijection of the set of vertices of $\FC(S_g)$.  Suppose that $\{c',d'\}$ is another bigon pair that determines $e$.  It follows from Corollary~\ref{cor:lka} that there is a sequence of bigon pairs
\[
\{c,d\} = \{c_0,d_0\} = \cdots = \{c_n,d_n\} = \{c',d'\}
\]
where each pair $\{\{c_i,d_i\},\{c_{i+1},d_{i+1}\}\}$ is a linked sharing pair for $e$.  It follows then from Proposition~\ref{prop:curve pair pairs} that $\hat \alpha$ is well defined on the vertices of $\FC(S_g)$.  Indeed, a consequence of Proposition~\ref{prop:curve pair pairs} is that if two bigon pairs form a sharing pair then their images under an automorphism of $\FC(S_g)$ determine the same inessential curve.

To complete the first step, we must show that $\hat \alpha$ is indeed an automorphism of $\FC(S_g)$, that is, it takes edges to edges.  For an edge spanned by two essential curves, this is automatic from the definition.  For an edge spanned by one essential curve $c$ and one inessential curve $e$, this follows from the fact that we can find a bigon pair that determines $e$ and is disjoint from $c$.  The case of an edge spanned by two inessential curves is similar.  

By definition the map $\varepsilon : \Aut \FC(S_g) \to \Aut \EFC(S_g)$ given by $\varepsilon(\alpha) = \hat \alpha$ is the desired extension map.  

\medskip

\noindent \emph{Step 2.} Recall that $\eta : \Homeo(S_g) \to \Aut \FC(S_g)$ and $\nu : \Homeo(S_g) \to \Aut \EFC(S_g)$ are the natural homomorphisms.  By Theorem~\ref{thm:efc}, the map $\nu$ is an isomorphism.  Let $\varepsilon$ be the extension homomorphism guaranteed by the first step.  We consider the composition:
\[
\Homeo(S_g) \stackrel{\eta}{\to} \Aut \FC(S_g) \stackrel{\varepsilon}{\to} \Aut \EFC(S_g) \stackrel{\nu^{-1}}{\to} \Homeo(S_g).
\]
We claim that this composition is the identity.  Indeed, since $\nu$ is the natural map, and since $\varepsilon$ is an extension homomorphism, it follows that for any $f$ we have 
\[
\eta \circ \nu^{-1} \circ \varepsilon \circ \eta(f) = \eta(f).
\]
Since $\eta$ is injective (Lemma~\ref{lemma:inj}) it follows that 
\[
\nu^{-1} \circ \varepsilon \circ \eta(f) = f,
\]
which is to say that $\nu^{-1} \circ \varepsilon$ is a left inverse to $\eta$.  The theorem follows.
\end{proof}

\bibliographystyle{plain}
\bibliography{autfc}

\end{document}